\documentclass{amsart}
\pdfoutput=1
\usepackage{hyperref}
\usepackage{ResMon,amssymb}
\usepackage{graphicx}
\usepackage{subfigure}
\usepackage{color}

\begin{document}
\title[On the equatorial Dehn twist of a Lagrangian nodal sphere]{On the equatorial Dehn twist of a Lagrangian nodal sphere}
\author{Umut Varolgunes}
\begin{abstract}
Let $(M^4,\omega)$ be a geometrically bounded symplectic manifold, and $L\subset M$ a Lagrangian nodal sphere such that $\omega\mid_{\pi_2(M,L)}=0$. We show that an equatorial Dehn twist of $L$ does not extend to a Hamiltonian diffeomorphism of $M$. We also confirm a mirror symmetry prediction about the action of a symplectomorphism extending an equatorial Dehn twist on the Floer theory of the nodal sphere. We present analogues of the equatorial Dehn twist for more singular Lagrangians, and make concrete conjectures about them.
\end{abstract}
\maketitle

\section{Introduction}

Let $(M,\omega)$ be a four dimensional symplectic manifold. We assume that $M$ is geometrically bounded (e.g. closed, completion of a domain with contact boundary, $\ldots$) in the sense of \cite{Gromanopen}. This assumption is necessary for being able to do Floer theory on $M$. We call $L\subset M$ a \textbf{Lagrangian nodal sphere} if it is an immersed Lagrangian $S^2$ with a single transverse self-intersection point. We make the assumption $\omega\mid_{\pi_2(M,L)}=0$ throughout the paper. We call the self-intersection point of a nodal sphere its double point.

Let us call a map $L\to L$ an \textbf{equatorial Dehn twist} if it can be obtained by doing a Dehn twist around a simple closed curve in $L$ which does not pass through the double point (see Figure 1). 

\begin{theorem}\label{thm1}
An equatorial Dehn twist on L does not extend to a Hamiltonian diffeomorphism of M.
\end{theorem}

We discuss two proofs of this theorem. The first one uses Polterovich's Lagrangian surgery \cite{Polterovichsurgery} to reduce Theorem \ref{thm1} to a similar result of Lalonde-Hu-Leclercq \cite{LalondeHu} for embedded Lagrangian submanifolds (see Theorem \ref{thmhu} in Section \ref{s1} for the complete statement). We note that the proof of their result utilizes Lagrangian Floer theory.

In the second approach, we analyze the action of a symplectomorphism $\Phi:M\to M$ that extends an equatorial Dehn twist on the Floer theory of the Lagrangian nodal sphere \cite{Akahoimmersed}. We stress here that this proof only works when $L$ represents a homology class whose self-intersection is zero, which is equivalent to the two branches at the double point having positive intersection in the sense of Whitney \cite{Whitney}. There are good reasons for this, see Corollary \ref{corklein}, Remark \ref{rmkklein}, and Remark \ref{rmklast}.

Let $A_{\infty}$-algebra $\mathcal{A}$ be a chain level model of the Floer algebra $CF(L,L)$ with  $\mathbb{C}$-coefficients. It is known that $\mathcal{A}$ is formal and that its cohomology is isomorphic to $\Lambda:=\Lambda^*(\mathbb{C}^2)$ \cite{Paulcategorical}.

The symplectomorphism $\Phi$ induces an $A_{\infty}$-quasi-isomorphism $\Phi_*:\mathcal{A}\to\mathcal{A}$, which is well defined only up to  homotopy conjugation (see Subsection \ref{ssact} for precise definitions and more details). The following statement is independent of this ambiguity.
\begin{theorem}\label{thm2} $\Phi_*$ acts as identity on $H(\mathcal{A})$, but it is not homotopic to the identity as an $A_{\infty}-$automorphism.
\end{theorem}

Note that Theorem \ref{thm2} implies Theorem \ref{thm1} using a straightforward generalization of Section 10c of \cite{Paulbook} to the immersed Floer theory (also see \cite{BaoAlstonimmersed}), showing that such $\Phi$ cannot be Hamiltonian isotopic to identity. We in fact have a more precise result concerning the second part of Theorem 2, which we state only in a sketchy way. For a more detailed version, see Subsection \ref{ssMC}. 

Let $FAut_0(\mathbb{C}^2)$ be the group of \textbf{formal automorphisms} of the plane $\mathbb{C}^2$, i.e. $<p,q>$-adically continuous $\mathbb{C}$-algebra isomorphisms $\mathbb{C}[[p,q]]\to \mathbb{C}[[p,q]]$. The reader can think of these as the isomorphisms that fix the ideal $<p,q>$ and are uniquely determined by the images of $p$ and $q$. We will see that we can define a group homomorphism:
\begin{equation}\label{symmetrization}
Aut_{A_{\infty}}(\Lambda)\to FAut_0(\mathbb{C}^2)
\end{equation}

The crucial point about the map (\ref{symmetrization}) is that it sends $A_{\infty}$-homotopic $A_{\infty}$-automorphisms to the same formal automorphism. Using the formality of $\mathcal{A}$, we can transfer $\Phi_*:\mathcal{A}_L\to\mathcal{A}_L$ to an $A_{\infty}$-automorphism $\Lambda\to\Lambda$, which is well defined up to homotopy and conjugacy. This way, we obtain a well defined conjugacy class in $FAut_0(\mathbb{C}^2)$. The following computation also implies the second part of Theorem \ref{thm2}.

\begin{theorem}\label{thm3} A representative of the conjugacy class associated to $\Phi_*$ is given by 
\begin{equation}\label{exponential}
\psi(p,q)=(pe^{pq},qe^{-pq}),
\end{equation}
where the exponential here is defined as the usual (formal) Taylor series expansion $e^s=1+s+\frac{s^2}{2}+\ldots$. 
\end{theorem}

\begin{figure}
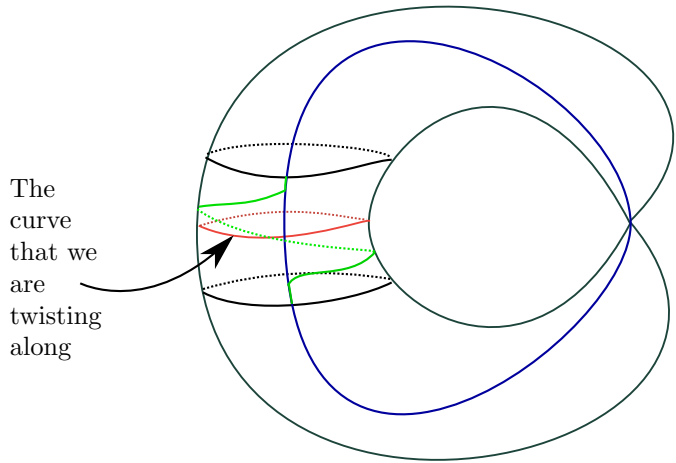
\caption{An equatorial Dehn twist. The curve we are twisting along is in red. The twist is supported in between the two black curves, and it maps the blue curve to the green one in there.}
\end{figure}

At this point an example of such $(M,L)$ pair is overdue. Consider the affine variety $\{(xy-1)z=1\}\subset \mathbb{C}^3$ with the standard symplectic structure induced from the embedding. Then the set $\abs{xy-1}=1, \abs{x}=\abs{y}$ describes a Lagrangian nodal sphere \cite{Paulcategorical}. It is easy to see that our theorems all apply in this case. For generalizations of this example see \cite{Alstonan}; and \cite{BaoAlstonimmersed} for why our theorems would apply for them even though the asphericity condition is not satisfied.

\subsection{Motivation}

The starting point for this study was a mirror symmetry conjecture.

Any Lagrangian nodal sphere with positive double point has a standard Weinstein neighborhood. This is a Weinstein domain with a Lagrangian nodal sphere as its Lagrangian skeleton. Let us call this manifold $W$, and the skeleton $Z$.

It has been long known in the mirror symmetry community that the mirror of $W$ is the complex manifold $\mathbb{C}^2-\{xy=1\}$. There are different constructions leading to it, and also non-trivial checks of the claim especially from a homological mirror symmetry viewpoint (see \cite{Paulcategorical} and the references therein, and also \cite{Shendecluster}). As an important point we mention that $L$ is mirror to the skyscraper sheaf at the origin in this conjectural story.

It has been conjectured that the standard lift of an equatorial Dehn twist of $L$ to a symplectomorphism of $W$ should be mirror to an actual complex automorphism $\psi$ of $\mathbb{C}^2-\{xy=1\}$, and that $\psi$ should preserve (setwise) the two cluster charts of $\mathbb{C}^2-\{xy=1\}$. A further investigation shows that there is not that many options for $\psi$, and the most likely one is \begin{equation}\label{cluster}
\psi(x,y)=(x(1-xy),y(1-xy)^{-1}).
\end{equation}

The final piece is that, invoking homological mirror symmetry, the formal automorphism that was mentioned in Theorem 3 then is supposed to be (see Remark \ref{rmksegal} for details) mirror to the action of $\psi$ on the formal neighborhood of the origin, which is simply given by:
\begin{equation}\label{formalcluster}
\psi(x,y)=(x(1-xy),y(1-xy+(xy)^2+\ldots)).
\end{equation}

To sum up, the precise conjecture is that the formal automorphism of Theorem 3 should be conjugate to this one, in other words, they should be related by a change of variables. It is elementary to see that this is the case (see Remark \ref{rmkmirror}).

We note that the results of this paper will be proven without resorting to any of these mirror symmetry considerations. 

\subsection{Technical details}

Let us start with a situation that is simpler than the one we actually need. Let $X$ be a smooth Spin manifold and $f:X\to X$ a diffeomorphism that preserves the spin structure. $f$ can be lifted to a symplectomorphism $F: T^*X\to T^*X$. It is well-known that there is a PSS type $A_{\infty}$-quasi-isomorphism (\cite{PSS},\cite{Abouzaidtopological}): \begin{equation*}
CF(X,X)\to\Omega_{dR}(X) , 
\end{equation*}
where by $X$ we also denote the zero section inside $T^*X$.
This isomorphism is defined in a canonical way, but the fact that the diagram \begin{align}\label{naturality}
\xymatrix{
 CF(X,X) \ar[d]^{F_*}\ar[r] &\Omega_{dR}(X) \ar[d]^{(f^{-1})^*}\\
CF(X,X) \ar[r] &\Omega_{dR}(X)}
\end{align} 
commutes up to homotopy, requires proof.

When $L$ is a Lagrangian nodal sphere in $(M^4,\omega)$ satisfying $\omega\mid_{\pi_2(M,L)}=0$, there exists a similar deRham type dga $\mathcal{A}_L$, called the Abouzaid model, and an $A_{\infty}$-quasi-isomorphism (see \cite{Abouzaidtopological}, and also \cite{Paulcategorical} for the version that we will be using):
\begin{equation*}
CF(L,L)\to \mathcal{A}_L.
\end{equation*}

An equatorial Dehn twist of $L$ acts on $\mathcal{A}_L$ by a version of pullback of differential forms. In the case where there exists a symplectomorphism $\Phi: M\to M$ extending the equatorial Dehn twist, there is the obvious analogue of diagram (\ref{naturality}) (see diagram (\ref{immersednaturality}) in Subsection \ref{ssact}), which again is expected to commute up to homotopy. 

Strictly speaking, we do not prove these naturality statements in this paper. We instead show that there exists a non-explicit $A_{\infty}$-quasi-isomorphism $CF(X,X)\to\Omega_{dR}(X)$ for which the diagram (\ref{naturality}) commutes up to homotopy. The philosophy here is the one of categorical actions as described in the Section 10b of \cite{Paulbook}. See Subsection 3.2 and the Appendix for more details.

There is another detail that we should mention here. In \cite{Abouzaidtopological}, in addition to the asphericity assumption, the Lagrangians were assumed to satisfy an exactness assumption. This was used to show that the computation of the Fukaya category can be done locally in a Weinstein neighborhood, by means of the integrated maximum principle (Lemma 7.3 in \cite{Paulbook}). Exactness assumption is in fact not necessary for the results of \cite{Abouzaidtopological} to hold, and locality can be achieved by either using virtual techniques \cite{FOOO}, or by a careful use of the monotonicity lemma \cite{Jcurves}. Both of these methods are well-known to the experts. We omit them in this paper. The skeptical reader can assume that, starting from Section 3, $M$ is exact, $\omega=d\theta$, and $L$ is strongly exact in the sense that $\theta$ vanishes on $H_1(L,\mathbb{R})$. Note that these exactness conditions hold for the example given after Theorem \ref{thm3}.

\subsection{Outline of the paper} In Section 2, we give the first proof of Theorem 1. We also give a neat corollary of Theorem 1 showing that the ``germ" of the aforementioned symplectomorphism of $W$ is exotic. We end the section with a conjectural generalization of this corollary to the ``Weinstein neighborhoods" of more general singular Lagrangians. In Section 3, we summarize the Floer theoretic aspects of the computation regarding Theorem \ref{thm3}, most importantly we introduce the Abouzaid model. We translate everything to a purely algebraic problem. In Section 4, first we give a precise version of Theorem \ref{thm3}, and then finish the computation. This requires some algebraic machinery, which is introduced in the most elementary way possible. In the Appendix, we sketch an approach to the technical Floer theoretical details we promised before.

We note that Section 2, and the portion that comes after it, are more or less independent of each other. Section 2 is elementary, but starting from Section 3 the reader is assumed to know the  basics of $A_{\infty}$-algebras (\cite{Kellerintro} would be enough), and also immersed Lagrangian Floer theory (since we always work in the simplest setting, \cite{Paulbook} along with \cite{Akahoimmersed} should be enough). The Lecture 11 of \cite{Paulcategorical} is a precursor to the second portion of the paper, and in particular all the references to \cite{Paulcategorical} are to that lecture unless otherwise stated. For the Appendix, the reader is assumed to be familiar with \cite{Abouzaidtopological} in addition to all this.

Here we note that all our $A_{\infty}$-algebras, maps etc. are $c$-unital. Another note is that when we want to think of a dga as an $A_{\infty}$-algebra, we define the structure maps by: $\mu_1(a)=(-1)^{\abs{a}}da,\mu_2(b,a)=(-1)^{\abs{a}}ba$, and $\mu_n=0, n>2$. This is necessary if we want to obey the sign conventions of \cite{Paulbook}, but also has some geometric content (see Definition 2.1 in \cite{Abouzaidtopological}). Nevertheless, we will always write our formulas in terms of the structures of the dga to avoid further confusion, which might make some signs look unusual.

\subsection{Acknowledgements} My first and foremost thanks go to my advisor Paul Seidel, for suggesting the problem, numerous enlightening discussions, and also reading a preliminary version of this paper, which resulted in many improvements. I thank Mohammed Abouzaid for a very useful discussion about the Appendix. I also thank Emmy Murphy, Roger Casals, Francesco Lin, and Sheel Ganatra for helpful conversations, and to Yanki Lekili and Semon Rezchikov for reading a portion of a preliminary version of the paper. This work was undertaken while the author was a PhD student at MIT. The author was supported by the NSF grants with award numbers 1265196 and 1500954.

\section{The first proof of Theorem 1, some corollaries and conjectures}\label{s1}

Let us first recall some definitions from the introduction. Let $(M,\omega)$ be a symplectic four-manifold. We call a subset $L$ of $M$ a Lagrangian nodal sphere if it is the image of a Lagrangian immersion $S^2\looparrowright M$ which brings two points of the sphere together transversely and is an embedding outside of those two points.  We make the assumption $\omega\mid_{\pi_2(M,L)}=0$. Finally, we call a map $L\to L$ an equatorial Dehn twist if it can be obtained by doing a Dehn twist, supported away from the double point, around a simple closed curve in $L$, which does not pass through the double point.

Before we go into the first proof of Theorem \ref{thm1}, we introduce the main ingredient, which is a theorem of Lalonde-Hu-Leclercq \cite{LalondeHu}. Note that the proof of this theorem uses Lagrangian Floer theory in a very essential way. One can write down a proof using only the formalism of Section 8 of \cite{Paulbook}, modulo some technical details to extend the results from the exact setting to the aspherical setting, but we skip this as most of the arguments end up being repetitions of \cite{LalondeHu}.

\begin{theorem}[Lalonde-Hu-Leclercq]\label{thmhu}
Let $K$ be a Lagrangian submanifold of $(M^{2n},\omega)$, and $\phi:K\to K$ be a diffeomorphism. Assume that $\omega\mid_{\pi_2(M,K)}=0$. If there exists a Hamiltonian diffeomorphism $\Phi: M\to M$ extending $\phi$, then $\phi^*:H^*(K,\mathbb{Z}_2)\to H^*(K,\mathbb{Z}_2)$ is the identity map. If $K$ is orientable and relatively spin, and $\phi$ preserves orientation and at least one relative spin structure, then the same statement with $\mathbb{Z}$-coefficients holds.
\end{theorem}

\begin{proof}[Proof of Theorem \ref{thm1}] Let us assume that there exists such  $\Phi :M\to M$. We denote the double point of $L$ inside $M$ by $p$.

We will use the following lemma repeatedly. Note that $V$ in the statement below will generally be a common domain of definition of Hamiltonian vector fields defined on different open neighborhoods of $p$, $C$ will be the closure of a smaller neighborhood, and $S$ will be the portion of $L$ lying in $V$.

\begin{lemma}\label{lemcutoff}
Let $F_t$ be a time dependent Hamiltonian vector field defined on a symplectic manifold $V$ that vanishes along a connected subset $S$ for all times. Let $C$ be any compact subset of $V$. Then, we can define a new time dependent Hamiltonian vector field $F'_t$ on $V$, which vanishes outside of an arbitrarily small open neighborhood of $C$, agrees with $F_t$ along $C$, and also still vanishes along $S$, for all times.
\end{lemma}
\begin{proof}
Because $S$ is connected, we can choose Hamiltonian functions $H_t$ for $F_t$ that vanish along $S$. We choose a new time dependent Hamiltonian of the form $\rho H_t$, where $\rho$ is a smooth bump function for any bounded open subset $U$ that contains $C$, i.e $\rho\mid_U=1$, and $\rho\mid_{V-U'}=0$, where $U'$ is any open subset such that $\bar{U}\subset U'$. The corresponding Hamiltonian vector field satisfies the conditions, because along $S$: $d(\rho H_t)=H_td\rho+\rho dH_t=0$.
\end{proof}

By a Hamiltonian isotopy $\{g_t\}_{t\in [0,1]}$ of symplectic embeddings $V\to Y$, where $V$ and $Y$ are equidimensional symplectic manifolds, we mean that associated the vector fields on $g_t(V)$ are Hamiltonian in their domain of definition. We also say that such an isotopy of embeddings (not necessarily Hamiltonian) is compactly supported, if there exists a compact subset $C$ of $V$ such that $g_t\mid_{V-C}$ is fixed for all times.

\textbf{Step 1:} $\Phi$ can be Hamiltonian isotoped to another diffeomorphism $\Phi'$ that still extends the equatorial Dehn twist, but also for which there exists an open ball $p\in B\subset M$ such that at all points of $B\cap L$, the Jacobian of $\Phi'$ is the identity.

By a complementary Lagrangian subbundle along $L$, we mean a Lagrangian subbundle of $TM\mid_{L-\{p\}}$, which is transversal to $T(L-\{p\})$, and fits together smoothly with the two tangent Lagrangian planes at $p$. Let $W$ be a Weinstein neighborhood model for $L$. For any complementary Lagrangian subbundle $\mathcal{L}$, we can construct a symplectic embedding from an open neighborhood of the nodal sphere in $W$ to $M$. Moreover, if we are given a one parameter family $\{\mathcal{L}_t\}_{t\in [0,1]}$, we get a Hamiltonian isotopy of embeddings.

We start with an arbitrary $\mathcal{L}$. Pushing forward by $\Phi$ produces another complementary Lagrangian subbundle $\Phi_*\mathcal{L}$. We can find a one parameter family $\{\mathcal{L}_t\}_{t\in [0,1]}$ such that $\mathcal{L}_0=\Phi_*\mathcal{L}$ and  $\mathcal{L}_1= \mathcal{L}$, becuse the space of complementary Lagrangian subspaces at each point is contractible. By the previous paragraph, we obtain a time dependent family of Hamiltonian vector fields defined on an open neighborhood of $L$. By construction, the time 1 map of this flow sends $\Phi_*\mathcal{L}$ to $\mathcal{L}$. Using Lemma \ref{lemcutoff} we cut these off to get a Hamiltonian isotopy supported near $p$, but defined in all of $M$, which  fixes $L$ pointwise, and still sends $\Phi_*\mathcal{L}$ to $\mathcal{L}$ near $p$. Since $\Phi\mid_L$, which is an equatorial Dehn twist, is identity near $p$, and a linear symplectomorphism that preserves two complementary Lagrangian subspaces and is identity on one of them has to be identity on the other as well, isotoping $\Phi$ using this Hamiltonian isotopy indeed gives the desired $\Phi'$.

\textbf{Step 2:} $\Phi'$ can be Hamiltonian isotoped to another diffeomorphism $\Phi''$ that still extends the equatorial Dehn twist, but also for which there exists an open ball $p\in B'\subset M$ such that all points of $B'$ are fixed under $\Phi''$.

First let us deal with the following local case. Namely, we consider $X:=\mathbb{R}^2\times\mathbb{R}^2$ with the symplectic structure $\omega$ obtained by its canonical identification with $T^*\mathbb{R}^2$. Let $K\subset X$ be the Lagrangian that is the union of the zero section, and the cotangent fiber of zero under that same identification. 

In what follows $U_i$, $i=1,2,3,4$, are open contractible neighborhoods of the origin in $X$, which shrink as $i$ increases. We take a symplectic embedding $f: U_1\to X$ that sends the points of $K\cap U_1$ to themselves, and moreover is the identity on $TX\mid_{K\cap U_1}$. Let us say that such an embedding (not necessarily symplectic) is fixed on $K\cap U_1$ to first order. 

We want to show that there is a compactly supported Hamiltonian isotopy $\{f_t\}_{t\in [0,1]}$ of symplectic embeddings $U_1\to X$ such that: \begin{enumerate} \item $f_0=f$ \item $f_1$ fixes an open neighborhood of the origin \item $f_t$ fixes $K\cap U_1$ pointwise. \end{enumerate}

We start with a compactly supported smooth isotopy $\{F_t\}_{t\in[0,1]}$ of embeddings $U_2\to X$ which satisfies properties $(1),(2)$ and a stronger version of $(3)$, where we require fixing to first order, not just pointwise. This isotopy can be constructed by taking convex linear combinations of $Id$ and $f$ near $p$, and using a parametric version of the inverse function theorem.

Now there exists a neighborhood $U_3$ of the origin: \begin{itemize} \item along which the family of forms $s\omega+(1-s)(F_t)^*\omega$ are all nondegenerate \item  on which $F_1$ is the identity\item which deformation retracts onto $K\cap U_3$.\end{itemize} For each $t$, we run the relative Moser argument for the family (varying with s) $s\omega+(1-s)(F_t)^*\omega$, and obtain $G_t$ defined in some neighborhood of the origin such that $(G_t)^*\omega=(F_t)^*\omega$, $G_0=G_1=id$, and $G_t$ fixes $K$ pointwise for all $t$. Here relative refers to the way in which we choose primitives, which actually is the standard one for proving neighborhood theorems in symplectic geometry, i.e. using the chain homotopy that we obtain from the deformation retraction (as in the Lemma 3.14 of \cite{McDuffSalamon}). Moreover, there is a common open domain of definition for $G_t$ by smooth dependence on initial data. We define the isotopy $\tilde{f_{t}}=fF_{1-t}^{-1}G_{1-t}$ which is symplectic and hence Hamiltonian. This defines a time dependent Hamiltonian vector field on $U_4$. Using Lemma \ref{lemcutoff}, we can extend it to $U_1$ by cutting it off. This gives us the desired $f_t$.

One then chooses local coordinates near $p$, and restricts $\Phi'$ to obtain an embedding of the form considered above. The compactly supported Hamiltonian isotopy constructed in the local model can then be implanted inside $M$, which achieves our goal.

\textbf{Step 3:} We can do Polterovich surgery to $L$, at $p$, to obtain an embedded Lagrangian $T$ such that $L\setminus T$ and $T\setminus L$ both lie in an arbitrarily small neighborhood of $p$, and moreover $\omega\mid_{\pi_2(M,T)}=0$.

Let us recall this surgery procedure very briefly (this reformulation is taken from \cite{AbouzaidSmith}). Again we work in the local model $X$ with the Lagrangian $K$. We also introduce complex numbers notation, namely we identify $\mathbb{C}^2_{z_1,z_2}\to \mathbb{R}^2_{x_1,x_2}\times \mathbb{R}^2_{y_1,y_2}$ using $z_j=x_j+iy_j$.

Let us now draw any proper smooth curve $\gamma: \mathbb{R}\to\mathbb{C}$ in the plane which satisfy the following conditions:
\begin{itemize}
\item There are no two numbers $t,t'\in\mathbb{R}$ such that $\gamma(t)=-\gamma(t')$. In particular, $\gamma$ doesn't pass through the origin
\item There exists $a>b\in\mathbb{R}$ such that $\gamma (t)$ is a positive real number for all $t\geq a$, and $\gamma (t)$ is a positive imaginary number for all $t\leq b$. 
\end{itemize} 
\begin{figure}
\def\svgwidth{0.5\linewidth}
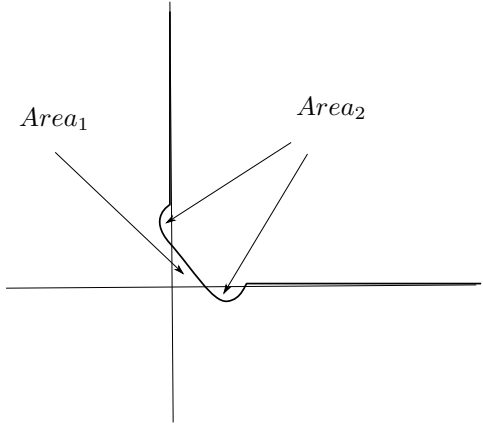
\caption{The curve that we use in the surgery procedure, where we require $Area_1=Area_2$}
\end{figure}

To any such $\gamma$, we can associate an embedded Lagrangian in $X$, which is given by $\bigcup \gamma(t)\cdot S^1$, where $S^1$ is the unit circle in $\mathbb{R}^2_{x_1,x_2}\times \{0\}$. Here $\cdot$ means scalar multiplication by complex numbers.

$L$ can also be modified by such $\gamma$ to obtain an embedded Lagrangian $T$. We claim that $\omega\mid_{\pi_2(M,T)}=0$ is satisfied, if the signed area between $\gamma$, and the piecewise linear curve $l$ that starts at $+\infty\cdot i$, goes to $0$ along the imaginary axis, and then continues in the positive direction of the real axis is zero. 

Consider the one parameter family of Lagrangians $\{L_s\}_{s\in [0,1]}$ such that $L_0=T$ and $L_1=L$, obtained by deforming $\gamma$ to $l$, keeping the area condition satisfied, and so that the first tangency points with the axes stay fixed for all times. We have canonical maps $L_{s'}\to L_s$ for $s'\leq s\in [0,1]$. For any disk $D$ with boundary on $T$, we can define a disk $D'$ with boundary on $L$, by adding $\bigcup_{s\in [0,1]} im(L_{0}\to L_s\mid_{\partial D})$ to $D$. 

We claim that $\int_{D}\omega=\int_{D'}\omega$, which implies the asphericity statement we want. For this, we need to show that the symplectic area of $\bigcup_{s\in [0,1]} im(L_{0}\to L_s\mid_{\partial D})$ is equal to $0$, which is a local computation depending only on $\partial D$. 

If $\partial D\cap \bigcup_{t\in [b,a]}\gamma(t)\cdot S^1$ consisted of horizontal lines (meaning lines of the form $\bigcup_{t\in [b,a]} \gamma(t)\cdot \alpha$, where $\alpha\in S^1$), then the claim would follow from the choice of $\gamma$. In fact, we can make this happen by isotoping $\partial D$ to another curve $\alpha$ on $T$. The final claim is that the areas of $\bigcup_{s\in [0,1]} im(L_{0}\to L_s\mid_{\partial D})$ and $\bigcup_{s\in [0,1]} im(L_{0}\to L_s\mid_{\alpha})$ are the same. 

We construct the two cycle formed by adding the isotopy between $\partial D$ and $\alpha$, and its transport to $L$, to these two (one with changed direction) chains. This two cycle is actually a boundary, as we can fill it in by the transports of the isotopy to other $L_t$'s. Because $T$ and $L$ are Lagrangians, the claim follows by Stokes theorem.

Clearly, the surgery region can be chosen arbitrarily small. 

\textbf{Step 4:} Get a contradiction using Theorem \ref{thmhu}. Note that in the positive double point case the result of the surgery $T$ from Step 3 is a torus, whereas in the negative double point case it is a Klein bottle. In both cases, $H^1(T,\mathbb{Z}_2)$ is two dimensional. We proved in Steps 1,2,3 that there exists a Hamiltonian diffeomorphism of $M$ that extends a diffeomorphism of $T$ which acts non-trivially on cohomology (a shear transformation in both cases). This is a contradiction to Theorem \ref{thmhu}.
\end{proof}

\begin{remark}\label{rmkZcoef}
In the positive double point case, we can instead use $\mathbb{Z}$-coefficients in Step 4 and obtain the theorem for any non-zero power of an equatorial Dehn twist. The only point to comment is that, in general, it is not true that every orientation preserving diffeomorphism $Y\to Y$ preserves at least one spin structure of $Y$ (not even if they act as identity on $H^1(Y,\mathbb{Z}_2)$, which is actually irrelevant). But, this statement is true for the two torus, in fact for any orientable surface. This follows from the fact that orientable 3-manifolds are Spin, by looking at the mapping torus of the diffeomorphism.
\end{remark}

\begin{remark}\label{rmktrivial}
The statement of this theorem would be wrong if we allowed diffeotopies instead of Hamiltonian isotopies. To see this note that any vector field on the sphere which is zero at the two points that come together can be extended to a Weinstein neighborhood of the nodal sphere.
\end{remark}

\subsection{Further directions}\label{ss1} Let us start with some elementary definitions. Let $W^4$ be an open symplectic manifold, and $Z\subset W$ be a two dimensional compact topological submanifold, which is a smooth Lagrangian submanifold away from finitely many points. We call such data a \textbf{geometric model}, and when it is clear what $Z$ is, we will denote it simply by the open manifold $W$.

Let $X^4$ also be a symplectic manifold. We define $LagGerm(Z,W,X)$ to be 
\begin{equation*}
\{(U \text{ an open neighborhood of $Z$ in $W$}, \text{symplectic embedding } \phi: U\to X)\}/\sim,
\end{equation*}
where $(U,\phi)\sim (U',\phi')$ if one has a $V\subset U,U'$, open neighborhood of $Z$, such that there exists symplectic embeddings $\{\psi_t:V\to X\}_{t\in [0,1]}$ with $\phi=\psi_0$ and $\phi'=\psi_1$  along $V$, and the total map $[0,1]\times V\to X$ being smooth. One can also define the smooth analogue $Germ(Z,W,X)$. 

Note that there is a special element in $LagGerm(Z,W,W)$ that is the germ of the identity map. Let us call that element trivial. If $\phi$ is a symplectic embedding that defines a germ in $LagGerm(Z,W,W)$, and fixes $Z$ setwise, by its \textbf{order} we mean the smallest positive integer $n$ such that the germ of $\phi^n$ is trivial in $LagGerm(Z,W,W)$. If there is no such $n$, we say that $\phi$ is of infinite order; if $n=1$, we call $\phi$ itself trivial. 

\begin{remark}
If the symplectic vector fields associated to some isotopy of symplectic embedding happen to be Hamiltonian in a germ of their domain of definition, then the Hamiltonians can be cut-off and extended to the whole ambient manifold making the flow globally defined. This comment is of course true without any conditions for $Germ$ and isotopy of smooth embeddings.
\end{remark}

As the sign of the intersection in the nodal sphere case shows $LagGerm(Z,W,X)$ can be sensitive to how $Z$ sits inside $W$, not just to $Z$, when $Z$ is not a smooth submanifold (obviously only on an open neighborhood of $Z$ inside $W$, but we choose to keep this notation). Not surprisingly, it also depends on $X$ (see the introduction of \cite{LalondeHu} for an example that is based on the Kodaira-Thurston manifold). 

\subsubsection{The torus case}\label{sss1}

Let us start with a corollary of Theorem 1 in the case where the nodal sphere has a positive double point. 

We first construct the geometric model $W$ corresponding to the Weinstein neighborhood of such Lagrangian nodal sphere as the self plumbing of the disc bundle inside $T^*S$ along the poles. See Figure 2 for a picture in the case of $T^*S^1$. We will think of $S$ as embedded in $\mathbb{R}^3_{x,y,z}$ in the standard way for ease of visualisation. Let $n$ and $s$ be the north and south poles. We take stereographic coordinates $(q_1^n,q_2^n)$ and $(q_1^s,q_2^s)$ corresponding to charts $\phi^n:B_{0.101}\to S$ and $\phi^s:B_{0.101}\to S$ around $n$ and $s$, where $B_{\epsilon}$ is the ball of radius $\epsilon$ in the $xy$-plane. This induces standard Darboux coordinates $(q_1^n,q_2^n,p_1^n,p_2^n)$ and $(q_1^s,q_2^s,p_1^s,p_2^s)$ inside $T^*S$ defined on the preimage of $\phi^n(B_{0.1})$ and $\phi^s(B_{0.1})$. Let $U$ be an open neighborhood of the zero section of $T^*S$ which intersects with these preimages exactly along $\{\abs{p^n}^2<0.1\}$ and $\{\abs{p^s}^2<0.1\}$. We can now define an open symplectic manifold $W$ by gluing $U$ to itself via the map: \begin{align}
( q^n, p^n) \mapsto (p^s,-q^s).\end{align}
Let us denote the canonical symplectic form on $W$ by $\omega_W$. By construction, we have a symplectic immersion $U\to W$, which characterizes $\omega_W$. The image of the zero section (which we denote by $Z$) inside $W$ is our geometric model.

Assume without loss of generality that the equatorial Dehn twist is done near the actual equator (i.e. intersection of the sphere with the $xy$-plane) and is supported at a small neighborhood $V\subset S$ of it. We take the $U$ above such that it contains the entire cotangent fibres above the points of $V$. Clearly, this equatorial Dehn twist can be extended to a symplectomorphism $g:W\to W$. 

\begin{figure}
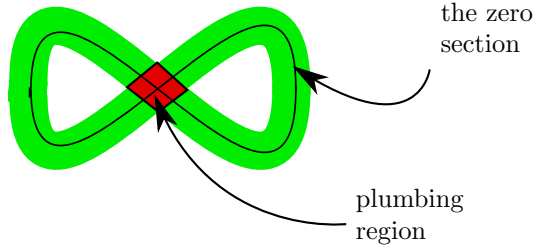
\caption{The plumbing construction for $T^*S^1$.}
\end{figure} 

\begin{corollary}\label{corflux}
$g$ is trivial in $Germ(Z,W,W)$, but it has infinite order in \\ $LagGerm(Z,W,W)$.
\end{corollary}

\begin{proof}
The smooth statement is easy as commented on before in Remark \ref{rmktrivial}. Let's assume that some power of $g$ is trivial in $LagGerm(Z,W,W)$. By an application of the symplectic flux technique \cite{McDuffSalamon}, we can construct from the isotopy of symplectic embeddings a Hamiltonian isotopy, which in the end extends the equatorial Dehn twist. This gives the desired contradiction by Remark \ref{rmkZcoef}. The only thing to note, since $H^1(W,\mathbb{R})=\mathbb{R}$, is the existence of a symplectic (but not Hamiltonian) vector field on $W$ whose flow exists for all times. This can be constructed by taking a function on $S$ which is equal to $0$ in a neighborhood of $s$, and $1$ in a neighborhood of $n$ such that these neighborhoods contain strictly the complement of $V$. This function can be pulled back to $T^*S$. Even though the function doesn't descend to $W$, its exterior derivative does, and the symplectic dual of this closed one-form gives us the desired vector field. Notice that this vector field is supported only away from the plumbing region, where the manifold looks like the portion of $T^*S$ near the equator. The flow there translates the cotangent fibres by the one form $d\rho$ on $S$. 

\end{proof}
\begin{remark}\label{rmklift}
It is easy to see that a diffeomorphism $S^2\to S^2$ fixing the poles extends to an element of $LagGerm(Z,W,W)$ if and only if $J_nJ_s^T=id$ where $J_n$ and $J_s$ are the Jacobian matrices at the poles in the $q$ coordinates we used in the construction. This implies that a rotation along the $z$-axis of $S^2$ extends to a symplectic embedding, whereas if the tangent spaces rotated in any other way in relation to each other it would not.
\end{remark}
\begin{remark}\label{rmknotdiff}
It is tempting to try to prove that $g$ as a diffeomorphism is smoothly isotopic to identity. In fact it is not even isotopic to a compactly supported diffeomorphism, which can be seen by looking at its action on the cohomology at infinity (as in \cite{Laitinen}). Also note that the careful choice of $U$ in the construction of $g$ is irrelevant for Corollary \ref{corflux}. 
\end{remark}

\subsubsection{The Klein bottle case}\label{sss2}

This immediately prompts the question of whether analogues of the previous subsection hold in the negative  double point case. There we can construct the geometric model $W'$ by changing the plumbing map to \begin{align}
( q^n_1,q^n_2, p^n_1,p^n_2) \mapsto( p^s_1,-p^s_2, -q^s_1,q^s_2).\end{align} Let us denote the analogous symplectomorphism by $g':W'\to W'$. 

\begin{corollary}\label{corklein}
$g'$ is trivial in $Germ(Z',W',W')$ but it has order two in \\ $LagGerm(Z',W',W')$ .
\end{corollary}

\begin{proof}
We show that the square of $g'$ is trivial in $LagGerm(Z',W',W')$. The other statements are completely analogous the positive double point case. The point is that now the condition in the Remark \ref{rmklift} has changed to $J_nRJ_s^TR=id$, where $R$ is reflection along the $\frac{\partial}{\partial q_2}$ axis. Hence, now the diffeomorphisms $S^2\to S^2$ that extend (if the Jacobian matrices at the poles are assumed to be rotations) to a germ are the ones that rotate the two poles in opposite directions. Here opposite term means opposite to the effect of the rotation along $z$-axis. Hence, one can see that in fact $g$ is equivalent as a germ to the standard lift of $180$ degrees rotation along the $z$-axis, from which the claim follows.
\end{proof}

\subsubsection{More singular Lagrangians}\label{sss3}

Let us now construct $W$ in a different way. Consider the solution set $\tilde{Z}$ of the equation $x^2=y^2$ inside $\mathbb{C}^2$. If we equip $\mathbb{C}^2$ with the symplectic structure $Re(dxdy)$, then $\tilde{Z}$ becomes an immersed Lagrangian. Now let us attach an open neighborhood of the zero section of $T^*(S^1\times (-\epsilon,1+\epsilon))$ to the unit ball in $\mathbb{C}^2$, along a Weinstein neighborhood of $Z$, near the boundary. More precisely,  
\begin{equation}
Z\cap \{(x,y)\mid 1-\epsilon<(\abs{x}^2+\abs{y}^2)^{1/2}<1\}
\end{equation}
can be identified with $S^1\times (-\epsilon,0)\cup S^1\times (1,1+\epsilon)$ such that the radius coordinate corresponds to the second coordinate. There are two different ways of gluing, which corresponds to the double point being positive or negative. It is enough to consider the positive one to make the point. We lift this to the desired symplectic gluing by Weinstein neighborhood theorem and the standard symplectic lift. This gives an open manifold $W$ with the nodal sphere $Z$, formed by the union of $\tilde{Z}$ with the core of the handle, inside (i.e. a geometric model). $W$ can be given a Weinstein structure but that seems irrelevant for our considerations.

Now, consider the diffeotopy of the unit ball: $(x,y)\mapsto (e^{2it}x, e^{2it}y)$, which fixes $Z$. When $t=\pi/2$, this gives a symplectomorphism. The diffeotopy can be extended to an open neighborhood of the core of the handle in such a way that $S^1\times (-\epsilon,1+\epsilon)$ rotates less and less from $0$ to $\epsilon$, does not rotate at all in $[\epsilon,1-\epsilon]$, starts rotating more and more from $1-\epsilon$ to $1$, and it matches the rotation of $Z$ in the gluing regions. Moreover this can be done in such a way that at $t=\pi/2$ the extension is still symplectic. This requires some care, a Moser argument is necessary in the gluing of two parts (as in the proof of Theorem \ref{thm1}) since the two (germs of) flows match up only up to Hamiltonian isotopy (on the handle we start with the standard lift of the smooth flow). The resulting element of $LagGerm(Z,W,W)$ is the same as the germ of $g$ above.

Clearly this construction can be generalized to the case where $x^2=y^2$ is replaced by $x^m=y^n$. W can also allow more general handle attachments to the link at infinity. The case where $(m,n)=1$ is the cleanest. There the link of the singularity has just one component, and we attach a disk to get the geometric model $W_{m,n}$ with Lagrangian $Z_{m,n}$ inside. The flow has to be replaced by $(x,y)\mapsto (e^{nit}x, e^{mit}y)$. Let us call $\phi_{m,n}$ the symplectic germ obtained at $t=\frac{2\pi}{m+n}$ (after extending to the handle, which is more canonical than the cylinder case by the Alexander trick).

\begin{conjecture}
$\phi_{m,n}$ is trivial in $Germ(Z_{m,n},W_{m,n},W_{m,n})$ but has order $m+n$ in $LagGerm(Z_{m,n},W_{m,n},W_{m,n})$.
\end{conjecture}

The hard part is to show that $\phi_{m,n}$ doesn't have lower order in $LagGerm$. In \cite{Shendecluster}, Shende et. al. construct non-isotopic exact Lagrangians inside $W_{m,n}$ and from their construction it seems like $\phi_{m,n}$ permutes those Lagrangians by a permutation which has order $m+n$. In the case of $(m,n)=(3,2)$, I, very recently, proved this statement by a more explicit understanding of these $5$ Lagrangians. The main geometric ingredient is to generalize Polterovich surgery with a combined use of possibly partial smoothings of complex singularities (e.g. cusp to node) and the hyperkahler trick, and to understand the exactness properties of these Lagrangians for a carefully chosen primitive. The method would likely generalize to all cases. The details of this will appear elsewhere.

\section{The abouzaid model}\label{s2}

\subsection{The model}\label{ssmodel}

From now until the end of the paper, we work in the setting where $L$ is Lagrangian nodal sphere with a positive double point inside a geometrically bounded symplectic manifold $M^4$ such that $\omega\mid_{\pi_2(M,L)}=0$.

$S^2$ is orientable and has a unique spin structure. It is well-known that in this case the $\mathbb{Z}_2$-graded Floer cochain groups $CF^*(L,L)$ with $\mathbb{Z}$-coefficients are well defined. Under the asphericity assumption we are making, this works almost identically to the embedded case (\cite{Akahoimmersed} \cite{BaoAlstonimmersed}).

In \cite{Abouzaidtopological}, Abouzaid introduced topological models for $CF^*(L,L)$. He used Morse and simplicial models. In \cite{Paulbook}, a deRham model was described, and that is the one we will use in this paper.

Let us again think of $S$ as embedded inside $\mathbb{R}^3$ as the standard unit sphere with the poles being the self-intersection points, and also let $D:=\{\abs{z}<1\}\subset\mathbb{C}$. Let $A:S\to S$ be the antipodal map. Denote the poles on $S$ by $n$ and $s$, and fix an embedding $\iota_n:D\to S$ onto a small neighborhood of $n$. Let $\iota_s:=A\circ \iota_n$. We note that we could choose any other embedding onto a neighborhod of $s$ with orientation different from $\iota_n$.

Abouzaid model $\mathcal{A}_L$ is a non-commutative dga. The underlying $\mathbb{Z}_2$ graded vector space is: \begin{equation*}
\Omega^*(S^2)\oplus\Omega^*(D)[-1]\oplus\Omega^*_{\text{cpct}}(D)[1].
\end{equation*}
We made a point of putting $-1$ shift to stress that a $\mathbb{Z}$ grading is algebraically possible. The differential is the deRham differential acting on each component separately.
The product structure is given by 
\begin{align*}
(\alpha_2,\beta_2,\gamma_2)*&(\alpha_1,\beta_1,\gamma_1)=(\alpha_2\alpha_1+(-1)^{\abs{\beta_1}}\iota_{n,*}(\gamma_2\beta_1)+(-1)^{\abs{\gamma_1}}\iota_{s,*}(\beta_2\gamma_1),\\&\iota_{n}^*(\alpha_2)\beta_1+(-1)^{\abs{\alpha_1}}\beta_2(\iota_{s}^*\alpha_1),\iota_{s}^*(\alpha_2)\gamma_1+(-1)^{\abs{\alpha_1}}\gamma_2(\iota_{n}^*\alpha_1)).
\end{align*}
\begin{theorem}[Abouzaid]\label{thmmodel}
There exists an $A_{\infty}$ quasi-isomorphism $\mathcal{A}_L\to CF^*(L,L)$. 
\end{theorem}

\begin{remark}\label{rmkchoices}
The proof of this theorem goes by comparing the two sides with a Morse model $\mathcal{M}_L$. The comparison of $\mathcal{M}_L$ with $\mathcal{A}_L$ is relatively standard differential topology, whereas that of $\mathcal{M}_L$ and $CF(L,L)$ is a little bit trickier, and involves setting up a PSS type moduli problem. In particular, one has to make a certain amount of choices for both sides to be defined (not to suggest that the former comparison is devoid of choices, for example a parametrization $S^2\looparrowright M$ of $L$ is needed in both). We will modify this construction in the Appendix to our needs so that we can prove the naturality statement we need. A priori it is not obvious if the quasi-isomorphism we use in the end is the same with the one constructed in \cite{Abouzaidtopological}.
\end{remark}

The proof of the following can be found in \cite{Paulcategorical}. We also write down an explicit $A_{\infty}$-quasi-isomorphism in the proof of Theorem 8, in Section 4.3.
\begin{proposition}[Abouzaid, Seidel]
$\mathcal{A}_L$ is formal, and its cohomology $H(\mathcal{A}_L)$ is isomorphic to $\Lambda:=\Lambda^*{\mathbb{C}^2}$.
\end{proposition}
\begin{remark}\label{rmkklein}
In the negative double point case, the underlying $\mathbb{Z}_2$ graded vector space is: \begin{equation*}
\Omega^*(S^2)\oplus\Omega^*(D)\oplus\Omega^*_{\text{cpct}}(D).
\end{equation*}
The only difference in the product structure is that the embeddings of the disk are both assumed to be orientation preserving. The formula is the same. Of course this drastically changes everything. The resulting dga is again formal and its homology is $\mathbb{C}[x,y]/(x^2,y^2)$, where everything has even degree. 
\end{remark}

\subsection{The action}\label{ssact}

Let us start with an algebraic definition. Let $\mathcal{A}$ be an $A_{\infty}$-algebra. We call two $A_{\infty}$-quasi-isomorphisms $f,g:\mathcal{A}\to\mathcal{A}$ \textbf{homotopy conjugate} if there exists  $A_{\infty}$-quasi-isomorphisms $h,h':\mathcal{A}\to\mathcal{A}$, which are homotopy inverses of each other, such that $f$ is homotopic to $h'gh$. In the c-unital context that we are working in, this is equivalent to the existence of a homotopy commutative diagram for a quasi-isomorphism $h:\mathcal{A}\to\mathcal{A}$:
 \begin{align}
\xymatrix{
\mathcal{A} \ar[d]^f \ar[r]^h &\mathcal{A} \ar[d]^g\\
\mathcal{A} \ar[r]^{h} &\mathcal{A}}
\end{align}.

If $\Psi:X\to X$ is a symplectomorphism which fixes an embedded Lagrangian $K$, then it gives a, well-defined up to homotopy conjugation, action on the Floer cochain groups of $K$, $\Psi_*: CF(K,K)\to CF(K,K)$. Let us spell this out in more detail. The strategy is to create an $A_{\infty}$-category $\mathcal{K}$ with $\mathbb{Z}$ many objects, denoted by $K_n$, where each object geometrically corresponds to the same Lagrangian $K$, and morphisms and structure maps are obtained by Floer theory. This requires a certain amount of choices, and the goal is to make those choices such that $\phi$ induces a strict $A_{\infty}$-functor $\mathcal{K}\to\mathcal{K}$, in other words, a $\mathbb{Z}$-action. Here strict means that only the first order term of the functor is allowed to be non-zero.

We follow Section 10b of \cite{Paulbook} pretty closely. We first make a universal choice of strip like ends. Then, we choose arbitrary Floer data for $CF(K_0,K_n)$, $n\in\mathbb{Z}$, and define the Floer data for $CF(K_i,K_{i+n})$ by pushing forward the chosen Floer data by $\Psi^i$. These give the morphism spaces in our category. Now we choose regular universal perturbation data for all $d>2$ pointed discs with integer labels such that the analogous equivariance is satisfied, by induction on $d$. More precisely, for each new $d$, we deal with the moduli spaces with the label $0$ to the left of the output first, satisfying the consistency and compatibility conditions, and then extend this to the other moduli spaces with $d$ marked points using powers of $\Psi$ (the conditions are automatically satisfied for those). We then define the structure maps of our category as in \cite{Paulbook}. 

Hence we constructed an $A_{\infty}$-category with:

\begin{itemize}
\item a strict $\mathbb{Z}$-action such that the orbit of one of the objects gives all the objects of the category
\item any two objects of the category are quasi-isomorphic
\end{itemize}

Let us call such a category an \textbf{$A_{\infty}$-groupoid with a strict transitive $\mathbb{Z}$-action}. Given such a category $\mathcal{K}$, any $A_{\infty}$-algebra that is quasi-isomorphic to the endomorphism $A_{\infty}$-algebra of an object of $\mathcal{K}$ is called an \textbf{algebra model}.

\begin{lemma}\label{lemmaaction}
Let $\mathcal{K}$ be an $A_{\infty}$-groupoid with a strict transitive $\mathbb{Z}$-action and $\mathcal{A}$ be an algebra model for $\mathcal{K}$. Then we obtain an $A_{\infty}$ quasi-isomorphism $\mathcal{A}\to\mathcal{A}$, that is well-defined up to homotopy conjugacy.
\end{lemma}

\begin{proof}
This follows immediately from Theorem 2.9 in \cite{Paulbook}, and the basic lemmas of the next subsection.
\end{proof}

\begin{remark}
For a carefully chosen model, one might be able to reduce the ambiguity to only homotopy. Yet, this seems less natural, and also not more helpful in terms of its practical use, as it seems unlikely that the statement of Theorem \ref{thmimm} can be made any stronger.
\end{remark}

This picture carries over to the context of the Lagrangian nodal sphere immediately because an equatorial Dehn twist acts trivially near the double point. Hence for $\Phi:M\to M$ extending an equatorial Dehn twist of $L$, we get $\Phi_*: CF(L,L)\to CF(L,L)$, well defined up to homotopy conjugation.

\begin{theorem}\label{thmimm}
There exists a homotopy commutative diagram,
 \begin{align}\label{immersednaturality}
\xymatrix{
\mathcal{A}_L \ar[d]^{(\phi^{-1})^*} \ar[r] &CF(L,L) \ar[d]^{\Phi_*}\\
\mathcal{A}_L \ar[r] &CF(L,L)}
\end{align} 
where the horizontal arrows are given by the same $A_{\infty}$-quasi-isomorphism.
\end{theorem}

\begin{proof}
See the Appendix.
\end{proof}

\begin{corollary}
The induced action $HF(L,L)\to HF(L,L)$ is trivial.
\end{corollary}
\begin{proof}
We remind the reader that there is nothing exotic about the differential on $\mathcal{A}_L$. With that in mind, one can easily find cocycles in $\mathcal{A}_L$ of which cohomology classes give a basis for homology and which are strictly fixed by $(\phi^{-1})^*$. 
\end{proof}

\subsection{Homotopy transfer}\label{sstransfer}

As was commented on before $\mathcal{A}_L$ is formal. Hence we have (non-canonical!) $A_{\infty}$-quasi-isomorphisms, $F:\mathcal{A_L}\to \Lambda$ and $G:\Lambda\to \mathcal{A}_L$, which are (two sided) homotopy inverses of each other. 

We then define the transfer quasi-isomorphism $\Lambda\to\Lambda$ to be $F(\phi^{-1})^*G$, i.e. as in the diagram:
\begin{align*}
\xymatrix{
\mathcal{A}_L \ar[d]^{(\phi^{-1})^*}  &\Lambda\ar[l]^G \ar[d]\\
\mathcal{A}_L \ar[r]^F &\Lambda}
\end{align*} 

More generally we have a map $Tr_{F,G}: Aut_{A_{\infty}}(\mathcal{A}_L)\to Aut_{A_{\infty}}(\Lambda)$.
\begin{lemma}
Assume that we are given a diagram:
\begin{align*}
\xymatrix{
\mathcal{A}\ar[d]^{T'}\ar[r]^{G'}&\mathcal{B} \ar[d]^{H} &\mathcal{A}\ar[l]^G \ar[d]^T\\\mathcal{A}&
\mathcal{B} \ar[r]^F\ar[l]^{F'} &\mathcal{A}}
\end{align*} 
such that the left and the right squares are homotopy commutative; and also $F$ and $F'$ are homotopy inverses of $G$ and $G'$ respectively, then $T$ is homotopy conjugate to $T'$.
\end{lemma}
\begin{proof} This follows immediately from preservation of homotopy under left and right compositions:
$T'\sim F'HG'\sim FG'F'HG'F'G\sim FHG\sim T$, where $\sim$ denotes homotopy conjugate.
\end{proof}
The following lemma is proved in a very similar manner.

\begin{lemma} $Tr_{F,G}$ preserves homotopy conjugacy.
\end{lemma}

For $\Lambda$ (or more generally any $A_{\infty}$-algebra with vanishing differential), one can talk about on the nose conjugacy because quasi-isomorphisms $\Lambda\to\Lambda$ admit unique strict (two sided) inverses.
\begin{lemma}
For $\Lambda$, homotopy conjugacy equivalence relation is the same as the one generated by homotopy and conjugacy.
\end{lemma}
\begin{remark}
It is easy to see that the latter equivalence relation can be more succintly described by $F$ and $G$ are equivalent to each other if $F$ is homotopic to a conjugate of $G$.
\end{remark}
\begin{corollary}
As a result of all these constructions, starting from $\Phi:M\to M$, we obtain a well-defined $A_{\infty}$ quasi-isomorphism $\Lambda\to \Lambda$, up to homotopy and conjugacy, for which the map on the first level is the identity.
\end{corollary}

\section{The formal automorphism of the plane}\label{s3}

Recall that we have the Abouzaid algebra $\mathcal{A}_L$, and the Dehn twist acting on $\mathcal{A}_L$ by pullback. We have also seen that this can be transferred to the minimal model $\phi:\Lambda\to \Lambda$, where $\phi$ is only well-defined up to conjugation and homotopy.

\subsection{Motivation}
We start with the following general situation. Let $f,g:\mathcal{B}\to\mathcal{B}$ be two $A_{\infty}$-automorphisms (i.e. self quasi-isomorphism) of an $A_{\infty}$-algebra $\mathcal{B}$. We want to ask the question of whether $f$ and $g$ are $A_{\infty}$-homotopic to each other. Induced maps on homology is one invariant that could work for this purpose, but as the situation we encountered in this paper shows it may not be enough.

We then go to the main wisdom of Morita theory, and consider modules over $\mathcal{B}$ to analyze the situation. The crucial fact here is that given any $A_{\infty}$-module over $\mathcal{B}$, we can pull it back by any $A_{\infty}$-automorphism and obtain another $A_{\infty}$-module. Moreover if we apply this procedure using $f$ and $g$, and they happen to be homotopic automorphisms, then the obtained modules turn out to be equivalent in the appropriate category. 

Hence one naive strategy would be to consider the moduli space of all $A_{\infty}$-modules over $\mathcal{B}$ up to equivalence (a formidable object), and show that $f$ and $g$ do not act exactly the same on this moduli space. Instead of analyzing the whole action what we end up doing mostly is to observe that $\mathcal{B}$ as an $A_{\infty}$-module over itself is a fixed point of both actions. Hence a more approachable canonical action to analyze is the one on the formal neighborhood of this fixed point.

This formal neighborhood is (tautologically) the space of formal deformations of $\mathcal{B}$ as an $A_{\infty}$-module over itself, and by the general premise of deformation theory it admits a very concrete description in terms of Maurer-Cartan elements, which we will use to make our computation.

Another important point that comes up in applications is the non-uniqueness of the model, in other words we might be in a situation where we want to do the computation in a quasi-isomorphic $A_{\infty}$-algebra $\mathcal{B}'$, and it often happens that we do not have too much control over how identify $\mathcal{B}$ and $\mathcal{B}'$. The formal neighborhood we described is suited well for this purpose. The choice of a quasi-isomorphism (really its homotopy class) between $\mathcal{B}$ and $\mathcal{B}'$ gives us an identification of the formal spaces, which we can use to transfer actions of automorphisms. The upshot is that the conjugacy class of the action of an automorphism on the formal neighborhood is a completely model invariant notion. As a concrete consequence of this observation we point out that in Theorem \ref{thmimm}, we do not need to (and in fact cannot at this point) specify anything about horizontal arrows, i.e. the identification.

\subsection{Maurer-Cartan elements of $\Lambda$}\label{ssMC}

Before we begin, we bring out a confusing point. We have been only assuming that our vector spaces are $\mathbb{Z}_2$-graded, but in fact $\mathcal{A}_L$, $\Lambda$, all the maps we have used, and will construct involving these two $A_{\infty}$-algebras can be made $\mathbb{Z}$-graded (nothing has to be changed in what we wrote). The $\mathbb{Z}$-grading of $\mathcal{A}_L$ is the one we alluded to when we first introduced it (Section 3.1), whereas the one of $\Lambda$ is the one with its generators having degree $1$. Moreover, the computation we are about to undertake makes sense and is exactly the same if we start using these $\mathbb{Z}$-gradings. We will do that from now on, only because it simplifies notation in a couple of places.

We will try to explain our computation staying in the exterior algebra $\Lambda$ as much as we can, since everything is much more elementary there. Notice though that we do not really know much about the transferred automorphism $\Lambda\to \Lambda$, but we have full information about $\mathcal{A}_L\to\mathcal{A}_L$. Hence, we will have to go get the information from there at some point.

For $\mathcal{B}=\Lambda$ or $\mathcal{A}_L$, we define a $\mathcal{B}$-MC-element to be an element $\alpha$ of $\mathcal{B}^1[[p,q]]$, which satisfies, \begin{equation*}d\alpha+\alpha\cdot\alpha=0,\end{equation*}and has constant term equal to zero. In case $\mathcal{B}=\Lambda$, the formal space of such elements is literally the formal neighborhood we talked about above, whereas if $\mathcal{B}=\mathcal{A}_L$, one needs to mod out by what is called the gauge action of $\mathcal{B}^0[[p,q]]$, which we do not go into here.

\begin{lemma}
All  $\Lambda$-MC elements are given by $f(p,q)a+g(p,q)b$ where $a$ and $b$ are a basis for the underlying vector space of $\Lambda$, and $f,g\in\mathbb{C}[[p,q]]$ with zero constant terms.
\end{lemma}

We call such a pair $f$ and $g$ a formal mapping of the plane by thinking of them as the corresponding adically continuous algebra homomorphism $\mathbb{C}[[p,q]]\to \mathbb{C}[[p,q]]$, given by $p\mapsto f(p,q)$, $q\mapsto g(p,q)$. Clearly, the invertible elements are the ones with invertible first order maps. We call such a pair a formal automorphism, and denote their set by $FAut_0(\mathbb{C}^2)$. Note that $FAut_0(\mathbb{C}^2)$ is a group, where the multiplication is given by composition. 

We can push-forward MC-elements, which confusingly corresponds to pulling-back modules. More precisely, if $\psi: \mathcal{B} \to \mathcal{B}'$ is an $A_{\infty}$-quasi-isomorphism, where both source and target are either $\Lambda$ or $\mathcal{A}_L$, and $\alpha$ is a $\mathcal{B}$-MC element, then \begin{align}
\psi_*\alpha :=\sum\psi_n(\alpha ,\ldots ,\alpha)\end{align}
is a $\mathcal{B}'$-MC element.

If we have an $A_{\infty}$-map $\phi:\Lambda\to \Lambda$, and an MC element $\alpha= f(p,q)a+g(p,q)b$, the push forward MC element admits a nice description: \begin{align*}
\phi_*\alpha :=& \sum\phi_n(\alpha ,\ldots ,\alpha)\\=&(\sum_{k,l}(\sum_{\sum a_i=k+l, a_0\geq 0}\phi_{k+l}(\underbrace{a,\ldots,a}_{a_0},\underbrace{b,\ldots,b}_{a_1},a\ldots)_a)f^kg^l)a+\\&(\sum_{k,l}(\sum_{\sum a_i=k+l,a_0\geq 0}\phi_{k+l}(\underbrace{a,\ldots,a}_{a_0},\underbrace{b,\ldots,b}_{a_1},a\ldots)_b)f^kg^l)b,
\end{align*}

The following lemma provides the basis of the argument, which is just a reformulation of the computation we just did.

\begin{lemma}\label{lemeasy}
Let $\phi:\Lambda\to \Lambda$ be an $A_{\infty}$-map. Let $\tilde{\phi}: \mathbb{C}[[p,q]]\to \mathbb{C}[[p,q]]$ be the formal mapping corresponding to the pushforward of $pa+qb$. Then the pushforward of any MC element $f(p,q)a+g(p,q)b$ is given by the composition: \begin{equation}
\tilde{\phi}(f(p,q),g(p,q)))a+\tilde{\phi}(f(p.q),g(p,q)))b
\end{equation}
\end{lemma}

We define the \textbf{symmetrization map} $S: Aut_{A_{\infty}}(\Lambda)\to FAut_0(\mathbb{C}^2)$, by $S(\phi)=\tilde{\phi}$. Its properties are summarized in the following theorem.

\begin{theorem}
\begin{enumerate}
\item $S$ is a group homomorphism.
\item If $\phi$ and $\phi'$ are homotopic as $A_{\infty}$-maps, then $S(\phi)=S(\phi')$.
\end{enumerate}
\end{theorem}

\begin{proof} The first part follows immediately from Lemma \ref{lemeasy}. For the second part, one first observes that the homotopy implies that the first order terms are the same. Then, a simple induction on degree using the formula for the pushforward we derived above finishes the proof as the homotopy relation is (schematically written):\begin{equation*}
f-g= f(\ldots)\wedge H(\ldots)-H(\ldots)\wedge g(\ldots)+\cdot\wedge H(\ldots)- H(\ldots)\wedge\cdot
\end{equation*}
\end{proof}

If $\alpha$ and $\beta$ are two $\mathcal{B}$-MC-elements, we can define a chain complex:
\begin{equation}
hom(\alpha,\beta):=(\mathcal{B}[[p,q]],D), \text{ where }D\gamma=d\gamma+ (-1)^{\abs{\gamma}}\alpha\cdot \gamma-\gamma\cdot\beta.
\end{equation}

Here is the quasi-isomorphism invariance statement that we will use. This is the one place in the main line of argument that we do not present a proof. 

\begin{lemma}\label{equalitycrit}
Let $\alpha$ and $\beta$ be two $\Lambda$-MC elements, and let $f,f':\Lambda\to \mathcal{A}_L$ be two homotopic quasi-isomorphisms. Then, $\alpha=\beta$ if and only if $H^0(hom(f_*\alpha,f'_*\beta))\neq 0$.
\end{lemma}

\begin{proof}
First note that the $\mathcal{B}=\Lambda$, $f=f'=id$ case is trivially true by explicit computation. Then we use section (3m) in \cite{Paulbook} to conclude in the general case.
\end{proof}

\begin{remark}\label{rmksegal}
Let us briefly explain how this formal automorphism story comes up in algebraic geometry. Let $Y^n$ smooth variety, and $\psi:Y\to Y$ be an automorphism which fixes a closed point $y$. By taking formal coordinates around $y$ this produces for us a formal automorphism $\mathbb{C}[[x_1,\ldots ,x_n]]\to \mathbb{C}[[x_1,\ldots ,x_n]]$, which is well-defined up to change of variables (i.e. conjugation). 

In \cite{Segalpoint}, Segal shows that this formal automorphism can also be interpreted in the following categorical way. Let $\mathcal{C}$ be the DG-ehnancement of $D^b(Coh(Y))$, which can be conceretely realized as the homotopy category of bounded perfect complexes in this case. A Koszul resolution $K$ of the skyscraper sheaf $O_y$ is an object of $\mathcal{C}$ and we can talk about its formal neighborhood in the moduli space of objects of $\mathcal{C}$. Manifestly $\psi$ acts on this formal neighborhood, and as shown in \cite{Segalpoint}, this action agrees with action from the first paragraph. The formal neighborhood is concretely given by $A_{\infty}$ (or dg) module deformations of $\mathcal{B}:=Hom_{\mathcal{C}}(K,K)$ as a module over itself. Note that $\mathcal{B}$ is also formal, and hence quasi-isomorphic to $Ext^*(O_y,O_y)=\Lambda$, as would be expected from mirror symmetry \cite{Paulcategorical}. Therefore, after choosing a quasi-isomorphism $\mathcal{B}\to\Lambda$ (akin to the choice of coordinates in the first paragraph), even more conceretely, this formal neighborhood is precisely the moduli space of MC elements that we talked about in this section.  The action can then alternatively be computed in a similar manner as above, i.e. by first transferring the action on $\mathcal{B}$ to $\Lambda$ and then computing its symmetrization map.

Obviously, it is the categorical interpretation that we use to make the mirror symmetry prediction. On the symplectic side, since we do not have an analogue of Segal's theorem, we also need to use the alternative computation technique.
\end{remark}

\subsection{The computation}\label{sscomputation}

Let us now state the problem again, with all terminology and tools introduced.

We have the homotopy commutative diagram \begin{align}
\xymatrix{
\mathcal{A}_L \ar[d]^{(\phi^{-1})^*}  &\Lambda\ar[l]^G \ar[d]^{Tr}\\
\mathcal{A}_L \ar[r]^F &\Lambda,}
\end{align} with $F$ and $G$ homotopy inverses of each other, and we want to compute $S(Tr)$. Note that only $S(Tr)$ up to change of variables is meaningful and we are free to choose $F$ and $G$ as we please.

\begin{theorem}
$F$ and $G$ can be chosen such that $S(Tr)$ is given by:
\begin{equation}
p\mapsto pe^{pq}, q\mapsto qe^{-pq},
\end{equation}
where the exponential here is defined as the usual (formal) Taylor series expansion $e^s=1+s+\frac{s^2}{2}+\ldots$.
\end{theorem}

\begin{proof}
\textbf{Step 1:} In order to choose $G$ let us describe some elements of $\mathcal{A}_L$:\begin{itemize}
\item $1:=(1,0,0)$
\item $x:=(0,1,0)$
\item Choose one $\eta\in\Omega^2_{cpct}(D)$ such that $\int_D\eta =1$ and let $y=(0,0,\eta)$.
\item Choose one $\tilde{\xi}\in\Omega^1(S)$ such that $d\tilde{\xi}=\iota_{n,*}\eta+\iota_{s,*}\eta$ and $A^*\tilde{\xi}=\tilde{\xi}$. Let $\xi=(\tilde{\xi},0,0)$. 
\end{itemize}

Note the important relationships:\begin{itemize}
\item $d\xi=-xy-yx.$
\item $\xi y= y\xi =0$
\item $\xi x+x\xi=0$
\end{itemize}

Now we define $G$ by $G_1(1)=1$, $G_1(a)=x$, $G_1(b)=y$, $G_1(a\wedge b)=-yx$, $G_2(a,b)=\xi$, $G_2(a,a\wedge b)=x\xi$. We declare that all the other ways of inputting the basis elements $1,a,b,a\wedge b$ into $G_n$'s give zero, and we extend by linearity. It's tedious but easy to check that this is an $A_{\infty}$-quasi-isomorphism, keeping in mind the sign conventions of changing from a dga to an $A_{\infty}$-algebra.

We define $F$ to be any homotopy inverse of $G$.

\textbf{Step 2:} Let $\alpha$ be the MC element $pa+qb$, and $\beta:=pe^{pq}a+qe^{-pq}b$. We want to show that $Tr_*\alpha=\beta$. We will use Lemma \ref{equalitycrit}. Clearly, $G:\Lambda\to \mathcal{A}_L$ and $(\phi^{-1})^*GTr^{-1}:\Lambda\to \mathcal{A}_L$ are homotopic. Hence it suffices to show that \begin{equation}\label{eqnhtpy}
H^0(hom(G_*\beta,((\phi^{-1})^*G)_*\alpha)\neq 0
\end{equation}

Before we do that we need to describe a couple more elements of $\mathcal{A}_L$.

\begin{itemize}
\item $\xi':=(\phi^{-1})^*\xi$
\item Using Stokes theorem two times, we see that if $\tilde{\rho}\in\Omega^0(S)$ is such that $d\tilde{\rho}= (\phi^{-1})^*\tilde{\xi}-\tilde{\xi}$, then $\tilde{\rho}(n)-\tilde{\rho}(s)=1$. We choose one such that $\tilde{\rho}(n)=1,$ and $\tilde{\rho}(s)=0$. Let $\rho=(\tilde{\rho},0,0).$
\end{itemize}

Note the relations:
\begin{itemize}
\item $\rho x=x, x\rho=0$
\item $y\rho=y, \rho y=0$
\end{itemize}

Let us now describe the two MC elements that appear in (\ref{eqnhtpy}). In fact it is straightforward to see that:
\begin{equation*}
G_*\beta=pe^{pq}x+qe^{-pq}y+pq\xi
\end{equation*}
\begin{equation*}
((\phi^{-1})^*G)_*\alpha=px+qy+pq\xi'
\end{equation*}

\textbf{Step 3:} We define the degree $0$ element: \begin{equation} \label{cochain}
e^{pq\rho}:=\sum \frac{(pq)^n\rho^n}{n!}=(\sum \frac{(pq)^n\tilde{\rho}^n}{n!},0,0)\in hom(G_*\beta,((\phi^{-1})^*G)_*\alpha).
\end{equation}

Obviously $e^{pq\rho}$ is not a coboundary, hence it is enough to show that it's a cocycle:
 \begin{align*}
De^{pq\rho}&=de^{pq\rho}+ (pe^{pq}x+qe^{-pq}y+pq\xi)e^{pq\rho}-e^{pq\rho}(px+qy+pq\xi') \\&=pqe^{pq\rho}d\rho +(pe^{pq}x+qy+pqe^{pq\rho}\xi)-(pe^{pq}x+qy+pqe^{pq\rho}\xi')\\&=pqe^{pq\rho}d\rho+pqe^{pq\rho}(\xi-\xi')=0.
\end{align*}

Therefore, we finished the proof of (\ref{eqnhtpy}). This implies by Lemma \ref{equalitycrit} that $\beta=Tr_*\alpha$, which by definition translates to the statement of the theorem.
\end{proof}

\begin{remark}
The choice of the cochain in Step 3 might look mysterious. Among the elements that we have defined so far really the only candidates are of the form $\sum_{i\geq 0} a_i(p,q)\rho^i$. If one goes backwards in the presented proof, (\ref{cochain}) comes out easily as the unique choice.
\end{remark}

\begin{remark}
If one is interested only in proving that $\Phi_*$, or equivalently $Tr:\Lambda\to\Lambda$, is not homotopic to identity (as in Theorem \ref{thm2}), there is an alternative proof, which uses more high-brow conceptual reasoning but is computationally more straightforward.

First of all, for this proof, one needs explicit formulas for both $F$ and $G$, which requires the use of homotopy transfer formulas from \cite{Marklformulas}. For an explicit choice of starting data, i.e a ``Hodge decomposition" in the sense of \cite{Marklformulas} , one can compute that $Tr_2=0$ for all inputs, and $Tr_3\neq 0$ (in addition to the a priori known $Tr_1=id$). This means that $Tr$ is the exponential of a Hochschild cocycle $\eta$ (as in the Lecture 8 of \cite{Paulcategorical}) with the first nontrivial part $\eta_3$ which is equal to $Tr_3$. If $Tr$ were to be homotopic to identity then $\eta$ would have to be a Hochschild coboundary, which in particular means that $Tr_3$ represents the trivial Hochschild cohomology class. Using Hochschild-Kostant-Rosenberg and also the explicit description of $Tr_3$ it is easy to find the cohomology class of $Tr_3$ (essentially given by symmetrization) and see that it is non-zero. This corresponds to the (conjugacy invariant!) fact that the first non-trivial coefficient in the formal diffeomorphism we computed is in degree 3, though it is tricky to show rigorously this link.
\end{remark}

\begin{remark}\label{rmkmirror}
In order to tie this back to the mirror symmetry prediction we have to show that $p\mapsto pe^{pq}, q\mapsto qe^{-pq}$ and $p\mapsto p(1-pq), q\mapsto qe(1-pq+p^2q^2\ldots)$ are conjugate to each other. The easiest way to see this is to observe that these maps are the time-$1$ map of the (formal) Hamiltonian vector fields of the Hamiltonians  $-(pq)^2/2$ and $-\frac{1}{1\cdot 2}(pq)^2+\frac{1}{2\cdot 3}(pq)^3-\frac{1}{3\cdot 4}(pq)^4+\ldots$ respectively. Hence any change of variables that sends $-(pq)^2/2\to  -\frac{1}{1\cdot 2}(pq)^2+\frac{1}{2\cdot 3}(pq)^3-\frac{1}{3\cdot 4}(pq)^4+\ldots$ would do the job, and it is easy to see that this exists. 
\end{remark}

\begin{remark} \label{rmklast}
This analysis can also be made in the case of a negatively self-intersecting nodal sphere. This will appear in a future paper. We just note that the action on the corresponding deformation space is trivial, even though one can in fact detect some non-triviality by looking at a nearby Klein bottle with a local system that has $-1$ monodromy along the equator. This suggests putting some extra data (analogous to the universal local systems in the case of embedded Lagrangians and Family Floer theory) for the nodal sphere to capture all the nearby information.
\end{remark}

\section{Appendix}

The aim of this appendix is to sketch a proof of Theorem \ref{thmimm}. We start with two easy algebraic lemmas about $A_{\infty}$-groupoids with transitive strict $\mathbb{Z}$-actions, and the induced automorphisms of their models. Recall that these notions were defined in Subsection 3.2.

Let $\mathcal{C}$ and $\mathcal{D}$ be two $A_{\infty}$-groupoids with transitive strict $\mathbb{Z}$-actions, and let $C$ and $D$ be a choice of models. Hence we have automorphisms $C\to C$ and $D\to D$, both well defined up to homotopy conjugation. The following follows from the naturality of the constructions before Theorem 2.9 in \cite{Paulbook}

\begin{lemma}\label{lemcom}
If we have a strictly $\mathbb{Z}$-equivariant quasi-equivalence $\mathcal{C}\to\mathcal{D}$, then there is a quasi-isomorphism $C\to D$ such that the diagram 
\begin{align*}
\xymatrix{
C \ar[d] \ar[r] &D\ar[d]\\
C \ar[r] &D}
\end{align*} 
commutes up to homotopy.
\end{lemma}

Let $B$ be an $A_{\infty}$-algebra with a given $A_{\infty}$-automorphism, $f:B\to B$. We construct an $A_{\infty}$-category $\mathcal{B}$ which has $\mathbb{Z}$ many objects, where all morphism spaces are equal to $B$ and structure maps are the same with the ones of $B$. $\mathcal{B}$ has a tautological transitive strict $\mathbb{Z}$ action, which induces $g: B\to B$ - as usual well-defined up to homotopy conjugation. The following is a tautology.

\begin{lemma}\label{lemtriv}
$f$ and $g$ are homotopy conjugate.
\end{lemma}

We first prove the naturality statement for embedded Lagrangians. Namely, let $K$ be an embedded Lagrangian inside a geometrically bounded symplectic manifold $X$ such that $\pi_2(X,K)=0$. Let $\psi :K\to K$ be a diffeomorphism and assume that $\Psi:X\to X$ is a symplectomorphism extending $\psi$.

As in Lemma \ref{lemmaaction}, if we take any model $\mathcal{A}$ for the Floer algebra of $K$, we get a well-defined (up to homotopy conjugation) action $\Psi_*:\mathcal{A}\to\mathcal{A}$.

\begin{theorem}\label{thmlast} There exists an $A_{\infty}$-equivalence $\mathcal{A}\to\Omega_{dR}(K)$ such that 
\begin{align*}
\xymatrix{
\mathcal{A} \ar[d]^{\Psi_*}\ar[r] &\Omega_{dR}(K) \ar[d]^{(\psi^{-1})^*} \\
\mathcal{A}\ar[r] &\Omega_{dR}(K) }
\end{align*} 
commutes up to homotopy.
\end{theorem}

\begin{proof} This goes by relating a chain of different models. More precisely, our road map will be Floer $\to$ Morse $\to$ Simplicial $\to$ Singular $\to$ DeRham. For each of these we prepare $A_{\infty}$-categories with transitive strict $\mathbb{Z}$-actions. This is done in a very similar manner to Section \ref{ssact} for the first three, and by the tautology explained right before Lemma \ref{lemtriv} for the last two. Recall that $\Psi_*$ was constructed using the Floer version of these categories. Then, we define strictly $\mathbb{Z}$-equivariant functors by again the same strategy, along with the constructions that are classical by now. For example, for Floer $\to$ Morse, we consider the usual PSS moduli spaces, but where we use Lagrangian labels for disk moduli spaces and labelings of the regions of the plane for gradient flow ones by the integers, encoding which Floer data, Morse functions, perturbation data etc are to be used. Again, we choose the required perturbation data inductively over the number of inputs, where we first deal with the case where the leftmost label is $0$ and then transfer those data to the other labelings by the symplectomorphism. After all these functors are constructed, Lemma \ref{lemcom} and \ref{lemtriv} gives the desired result. Note that the direction of the arrow in the road map may not match with the direction of the more natural functor as in Singular $\to$ DeRham.
\end{proof}

In \cite{Abouzaidtopological}, which mainly dealt with two transversely (or cleanly) intersecting embedded Lagrangians, the case of an immersed Lagrangian with transverse double points was only mentioned in the introduction. The generalization is straightforward, but there are two points that needs explanation. 

For the Morse model introduced in \cite{Abouzaidtopological}, to define the $A_{\infty}$-structure maps, labeled planar trees were used. Labeling meant assigning $0$ or $1$ to all the connected components of the complement of the planar tree. This then was used to tell the edges of the tree which gradient flow it should follow. The labeling of regions in this case is completely determined from where the inputs of the structure maps come. It might appear that in our case  there is not enough data to do the labeling. This is not true, because we definitely have the data to label which input edges make switches from $0$ to $1$, or $1$ to $0$. This determines uniquely a $0$ or $1$ labeling of all the regions, except when there is no switch in the inputs, in which case a labeling would be extraneous anyways. 

In order to use the analysis of \cite{Abouzaidtopological} for mushroom maps directly, we make a small modification to the construction of the Floer algebra of an immersed Lagrangian with transverse self intersections, assuming asphericity as usual. Let us describe this for the Lagrangian nodal sphere $L$ (the generalization will be clear). In the usual construction, in order to define $CF(L,L)$, one first chooses a Hamiltonion isotopy $f_t$ of $M$, making $L$ and $f_1(L)$ transverse, and then declares that the time 1 Hamiltonian chords of $f_t$ starting and ending at $L$ are the generators of the vector space. Assume that $f_t(L)$ stays sufficiently $C^1$ close to $L$.

The four chords are of two kinds: the two that come from the double point, which we call strange chords, and the other two that we call true chords (the ones that would persist if the isotopy were to be pulled back to $T^*S$). Note that if $f$ is small enough, there are no pseudo-holomorphic strips that involve the strange chords \cite{Akahoimmersed}. What we want to do is to replace the strange generators with two generators both corresponding to the actual double point, but labeled with the two possible branch jumps. This defines a new vector space (which is canonically identified with the old one, as the strange chords also have the property of doing branch jumps in different directions), with the same differential. Then, when we define higher structure maps, if one of these new generators occur in one of the marked points, we use zero Hamiltonians in the corresponding strip like ends (in particular the curve converges to the double point at that strip like end), and moreover, we require that the lift of the boundary map to $S^2$ changes branches in the direction labeled by the new generator (something that happened automatically for the corresponding strange chord). 

One can either take this as the definition of $CF(L,L)$, or, seeing the double point generators as a limit of strange generators of a carefully chosen one parameter family of Hamiltonian perturbations, show that this new $A_{\infty}$-algebra is quasi-isomorphic to the usual $CF(L,L)$, using parametrized moduli spaces. But, if we choose the second option, we also have to show that the quasi-isomorphism respects the action of $\Phi$. This requires us to again go back and define (or remember that we already did that) the actions through $A_{\infty}$-groupoids with transitive strict $\mathbb{Z}$-actions, construct strictly $\mathbb{Z}$-equivariant functors using parametrized moduli spaces, and replace the quasi-isomorphism with the non-explicit one given by algebra.

When all of this is done, then the argument in \cite{Abouzaidtopological} needs no modification. We choose the Floer data and Morse functions in our categories to satisfy the consistency condition (5.11) from \cite{Abouzaidtopological}. The paragraph about the Morse model applies to Lagrangian labelings (of pseudo-holomorphic discs and mushroom maps) as well, because we know which generator coming from the double point should be considered a $0$ to $1$ switch (and vice versa). As explained right after (5.11), this lets us make the required $\mathbb{Z}_2$ labeling of the stem and the cap of a mushroom consistently.

Note that in addition to this $\mathbb{Z}_2$-labeling, which is pretty crucial in obtaining compactness of moduli spaces (see Lemmas 2.5 and 5.18 in \cite{Abouzaidtopological}), we will also have an extra, harmless, $\mathbb{Z}$ labeling in what follows, completely analogous to the embedded case considered in the theorem above. These two labelings are totally independent of each other.

\begin{proof}[Proof of Theorem \ref{thmimm}]We follow exactly the same strategy with the proof of Theorem \ref{thmlast}. Using the first description of the Simplicial model in \cite{Abouzaidtopological}, it is easy to construct a Singular model. With a little bit more work in  Steps 1 and 2 of the proof of Theorem \ref{thm1}, we can show that $\Phi$ can be Hamiltonian isotoped so that it looks like the standard symplectic lift of the equatorial Dehn twist near all of $L$, not just near the double point. After this modification, Floer $\to$ Morse step is exactly the same, where we replace the PSS construction with the more complicated construction of \cite{Abouzaidtopological} of course. Morse $\to$ Simplicial is also a little bit tricky, but again this was done in \cite{Abouzaidtopological}, and we only need to add the $\mathbb{Z}$-labelings. The rest of the functors are constructed without any major additional complications to their embedded counterparts.
\end{proof}

\bibliographystyle{plain}

\bibliography{Nodalspherebib}

\end{document}